\documentclass[reqno,b5paper]{amsart}

\usepackage{amsmath}
\usepackage{amssymb}
\usepackage{amsthm}
\usepackage{enumerate}
\usepackage[mathscr]{eucal}
\usepackage{eqlist}

\theoremstyle{plain}
\newtheorem{theorem}{Theorem}[section]

\newtheorem{lemma}[theorem]{Lemma}

\newtheorem{claim}{Claim}
\newtheorem{note}[theorem]{Note}
\newtheorem{standard}[theorem]{Standard Lemma}

\theoremstyle{definition}

\numberwithin{equation}{section}
\numberwithin{equation}{section}
\begin{document}
\title[Additivity of maps preserving products $AP\pm PA^{*}$]%
{Additivity of maps preserving products $AP\pm PA^{*}$ on
$C^{*}$-algebras}
\author[A. Taghavi, V. Darvish and H. Rohi]%
{Ali Taghavi*, Vahid Darvish and Hamid Rohi}

\newcommand{\acr}{\newline\indent}
\address{\llap{*\,}Department of Mathematics,\\ Faculty of Mathematical
Sciences,\\ University of Mazandaran,\\ P. O. Box 47416-1468,\\
Babolsar, Iran.} \email{taghavi@umz.ac.ir,
v.darvish@stu.umz.ac.ir,h.rohi@stu.umz.ac.ir }


\begin{abstract}
Let $\mathcal{A}$ and $\mathcal{B}$ be two prime $C^{*}$-algebras.
In this paper, we investigate the additivity of map $\Phi$ from
$\mathcal{A}$ onto $\mathcal{B}$ that are bijective unital and
satisfies
$$\Phi(AP+\lambda PA^{*})=\Phi(A)\Phi(P)+\lambda \Phi(P)\Phi(A)^{*},$$
for all $A\in\mathcal{A}$ and $P\in\{P_{1},I_{\mathcal{A}}-P_{1}\}$
where $P_{1}$ is a nontrivial projection in $\mathcal{A}$ and
$\lambda\in\{-1,+1\}$. Then, $\Phi$ is $*$-additive.
\end{abstract}

\maketitle

\section{Introduction}\label{intro}
Let $\mathcal{R}$ and $\mathcal{R^{'}}$ be rings. We say the map
$\Phi: \mathcal{R}\to \mathcal{R^{'}}$ preserves product or is
multiplicative if $\Phi(AB)=\Phi(A)\Phi(B)$ for all $A, B\in
\mathcal{R}$. The question of when a product preserving or
multiplicative map is additive was discussed by several authors, see
\cite{mar} and references therein. Motivated by this, many authors
pay more attention to the map on rings (and algebras) preserving the
Lie product $[A,B]=AB-BA$ or the Jordan product $A\circ B=AB+BA$
(for example, Refs.
ref~\cite{bai,bai2,bre,hak,ji,lu,lu2,lu3,mol2,qi}). These results
show that, in some sense, the Jordan product or Lie product
structure is enough to determine the ring or algebraic structure.
Historically, many mathematicians devoted themselves to the study of
additive or linear Jordan or Lie product preservers between rings or
operator algebras. Such maps are always called Jordan homomorphism
or Lie homomorphism. Here we only list several results
\cite{bei,bei2,bei3,her,jac,kad,mar,mir,mir2}.

Let $\mathcal{R}$ be a $*$-ring. For $A,B\in\mathcal{R}$, denoted by
$A\bullet B=AB+BA^{*}$ and $[A,B]_{*}=AB-BA^{*}$, which are two
different kinds of new products. This product is found playing a
more and more important role in some research topics, and its study
has recently attracted many author's attention (for example, see
\cite{bre2,mol,sem}). A natural problem is to study whether the map
$\Phi$ preserving the new product on ring or algebra $\mathcal{R}$
is a ring or algebraic isomorphism. Obviously, if $\Phi$ is linear
or preserving star operation, then $\Phi$ preserves this new product
if and only if it preserves Lie product. Without the linearity and
star-preserving assumptions, what will occur? In \cite{cui}, J. Cui
and C. K. Li proved a bijective map $\Phi$ on factor von Neumann
algebras which preserves the new product ($[A,B]_{*}$) must be a
$*$-isomorphism. Moreover, in \cite{li} C. Li et al, discussed the
nonlinear bijective mapping preserving the new product ($A\bullet
B$). They proved that such mapping on factor von Neumann algebras is
also $*$-ring isomorphism. These two articles discussed new products
for arbitrary operators on factor von Neumann algebras.\\
In this paper, we will discuss such a bijective unital map (not
necessarily linear) on prime $C^{*}$-algebra which preserving both
new products for which one of them is projection must be
$*$-additive (i.e., additive and star-preserving).

Let $\mathbb{R}$ and $\mathbb{C}$ denote respectively the real field
and complex field and for real part and imaginary part of an
operator $T$ we will use $\Re(T)$ and $\Im(T)$, respectively. It is
well known that $C^{*}$-algebra $\mathcal{A}$ is prime, in the sense
that $A\mathcal{A}B=0$ for $A,B\in \mathcal{A}$ implies either $A=0$
or $B=0$.
\section{Main Results}
We need the following lemmas for proving our Main Theorem.
\begin{lemma}\label{0}
Let $\mathcal{A}$ and $\mathcal{B}$ be two $C^{*}$-algebras with
identities and $\Phi:\mathcal{A}\to\mathcal{B}$ be a unital map
which satisfies $\Phi(AP-PA^{*})= \Phi(A)\Phi(P)-\Phi(P)\Phi(A)^{*}$
for all $A\in\mathcal{A}$ and some $P\in\mathcal{A}$, then
$\Phi(0)=0.$
\end{lemma}
\begin{proof}
Let $A=I$, we have
$\Phi(0)=\Phi(IP-PI)=\Phi(I)\Phi(P)-\Phi(P)\Phi(I)^{*}$. Since
$\Phi$ is unital, we have $\Phi(0)=0$.
\end{proof}

\begin{standard}
Let $\mathcal{A}$ and $\mathcal{B}$ be two $C^{*}$-algebras and
$\Phi:\mathcal{A}\to\mathcal{B}$ be a map which satisfies
$\Phi(AP+\lambda PA^{*})= \Phi(A)\Phi(P)+\lambda\Phi(P)\Phi(A)^{*}$
for all $A\in\mathcal{A}$ and some $P\in\mathcal{A}$ where
$\lambda\in\{-1,+1\}$. Let $A, B$ and $T$ be in $\mathcal{A}$ such
that $\Phi(T)=\Phi(A)+\Phi(B)$. Then we have
\begin{equation}\label{1a}
\Phi(TP+PT^{*})=\Phi(AP+PA^{*})+\Phi(BP+PB^{*}),
\end{equation}
\begin{equation}\label{2a}
\Phi(TP-PT^{*})=\Phi(AP-PA^{*})+\Phi(BP-PB^{*}).
\end{equation}
\end{standard}
\begin{proof}
We will just prove the equality (\ref{1a}).\\
Multiply the equalities $\Phi(T)=\Phi(A)+\Phi(B)$ and
$\Phi(T)^{*}=\Phi(A)^{*}+\Phi(B)^{*}$ by $\Phi(P)$ from the right
and the left, respectively. We get
$$\Phi(T)\Phi(P)=\Phi(A)\Phi(P)+\Phi(B)\Phi(P),$$
and
$$\Phi(P)\Phi(T)^{*}=\Phi(P)\Phi(A)^{*}+\Phi(P)\Phi(B)^{*}.$$
By adding two equations, we have
$$\Phi(TP+PT^{*})=\Phi(AP+PA^{*})+\Phi(BP+PB^{*}).$$
\end{proof}
Our main theorem is as follows:
\\
\\
\textbf{Main Theorem.} Let $\mathcal{A}$ and $\mathcal{B}$ be two
prime $C^{*}$-algebras with $I_{\mathcal{A}}$ and $I_{\mathcal{B}}$
the identities of them, respectively. If $\Phi : \mathcal {A}\to
\mathcal{B}$ is a bijective unital map which satisfies
$\Phi(AP+\lambda PA^{*})= \Phi(A)\Phi(P)+\lambda\Phi(P)\Phi(A)^{*}$
for all $A\in\mathcal{A}$ and $P\in\{P_{1},I_{\mathcal{A}}-P_{1}\}$
where $P_{1}$ is a nontrivial projection in $\mathcal{A}$ and
$\lambda\in\{-1,+1\}$. Then, $\Phi$ is $*$-additive.
\\
\\
\textit{Proof of Main Theorem.} Let $P_{2}=I_{\mathcal{A}}-P_{1}$.
Denote $\mathcal{A}_{ij}=P_{i}\mathcal{A}P_{j},\ i,j=1,2,$ then
$\mathcal{A}=\sum_{i,j=1}^{2}\mathcal{A}_{ij}$. For every
$A\in\mathcal{A}$ we may write $A=A_{11}+A_{12}+A_{21}+A_{22}$. In
all that follows, when we write $A_{ij}$, it indicates that
$A_{ij}\in\mathcal{A}_{ij}$.\\
For showing additivity of $\Phi$ on $\mathcal{A}$ we will use above
partition of $\mathcal{A}$ and give some claims that prove $\Phi$ is
additive on each $\mathcal{A}_{ij}, \ i,j=1,2$.

\begin{claim}\label{1112}
For every $A_{11}\in\mathcal{A}_{11}$ and  $B_{12}\in
\mathcal{A}_{12}$, we have
$$\Phi(A_{11}+B_{12})=\Phi(A_{11})+\Phi(B_{12}).$$
\end{claim}
Since $\Phi$ is surjective, we can find an element
$T=T_{11}+T_{12}+T_{21}+T_{22}\in\mathcal{A}$ such that
\begin{equation}\label{f1}
\Phi(T)=\Phi(A_{11})+\Phi(B_{12}),
\end{equation}
we should show $T=A_{11}+B_{12}$. We apply the standard lemma
(\ref{1a}) to (\ref{f1}) for $P_{1}$, then we can write
$$\Phi(TP_{1}+P_{1}T^{*})=\Phi(A_{11}P_{1}+P_{1}A_{11}^{*})+\Phi(B_{12}P_{1}+P_{1}B_{12}^{*}),$$
so,
$$\Phi(T_{11}+T_{21}+T_{11}^{*}+T_{21}^{*})=\Phi(A_{11}+A_{11}^{*}),$$ by injectivity of $\Phi$, we get
$2\Re(T_{11}+T_{21})=2\Re(A_{11})$ which implies that
$\Re(T_{11})=\Re(A_{11})$ and $\Re(T_{21})=0$. Similarly, we apply
the standard lemma (\ref{2a}) to (\ref{f1}) for $P_{1}$, we have
$\Im(T_{11})=\Im(A_{11})$ and $\Im(T_{21})=0$.
Hence, we can say $T_{11}=A_{11}$ and $T_{21}=0$.\\
Now, we apply the standard lemma (\ref{1a}) to (\ref{f1}) for
$P_{2}$, it follows
\begin{eqnarray*}
\Phi(TP_{2}+P_{2}T^{*})&=&\Phi(A_{11}P_{2}+P_{2}A_{11}^{*})+\Phi(B_{12}P_{2}+P_{2}B_{12}^{*})\\
&=&\Phi(B_{12}+B_{12}^{*}).
\end{eqnarray*}
So, we have $\Re(T_{12})=\Re(B_{12})$ and $\Re(T_{22})=0$.\\
Similarly, we apply the standard lemma (\ref{2a}) to (\ref{f1}) for
$P_{2}$, we will have $\Im(T_{12})=\Im(B_{12})$ and $\Im(T_{22})=0$
which implie $T_{12}=B_{12}$ and $T_{22}=0$. So, $T=A_{11}+B_{12}$.

\begin{claim}\label{1221}
For every $A_{12}\in\mathcal{A}_{12}$, $B_{21}\in\mathcal{A}_{21}$,
we have
$$\Phi(A_{12}+B_{21})=\Phi(A_{12})+\Phi(B_{21}).$$
\end{claim}
Let $T=T_{11}+T_{12}+T_{21}+T_{22}\in\mathcal{A}$ be such that
\begin{equation}\label{b}
\Phi(T)=\Phi(A_{12})+\Phi(B_{21}).
\end{equation}
By applying the standard lemma (\ref{1a}) to (\ref{b}) for $P_{1}$,
we have
\begin{eqnarray*}
\Phi(TP_{1}+P_{1}T^{*})&=&\Phi(A_{12}P_{1}+P_{1}A_{12}^{*})+\Phi(B_{21}P_{1}+P_{1}B_{21}^{*})\\
&=&\Phi(B_{21}+B_{21}^{*}).
\end{eqnarray*}
Thus, $\Phi(2\Re(T_{11}+T_{21}))=\Phi(2\Re(B_{21}))$ which implies $\Re(T_{21})=\Re(B_{21})$ and $\Re(T_{11})=0$.\\
Similarly, we can obtain $\Im(T_{21})=\Im(B_{21})$ and
$\Im(T_{11})=0$ by applying (\ref{2a}) to (\ref{b}) for $P_{1}$.\\
Now, we apply the standard lemma (\ref{1a}) to (\ref{b}) for
$P_{2}$, it follows
$$\Phi(TP_{2}+P_{2}T^{*})=\Phi(A_{12}P_{2}+P_{2}A_{12}^{*})+\Phi(B_{21}P_{2}+P_{2}B_{21}^{*})=\Phi(A_{12}+A_{12}^{*}).$$
We have $\Phi(2\Re(T_{12}+T_{22}))=\Phi(2\Re(A_{12}))$, then
$\Re(T_{12})=\Re(A_{12})$ and $\Re(T_{22})=0$. In a similar way, by
applying the standard lemma (\ref{2a}) to (\ref{b}) for $P_{2}$, we
can obtain $\Im(T_{12})=\Im(A_{12})$ and $\Im(T_{22})=0$. Therefore
we have the result.

\begin{claim}\label{ijij}
For every $A_{ij}, B_{ij}\in \mathcal{A}_{ij}$ such that $1\leq
i\neq j\leq 2$, we have
$$\Phi(A_{ij}+B_{ij})=\Phi(A_{ij})+\Phi(B_{ij}).$$
\end{claim}
Let $T=T_{11}+T_{12}+T_{21}+T_{22}\in \mathcal{A}$ be such that
\begin{equation}\label{bbb2}
\Phi(T)=\Phi(A_{ij})+\Phi(B_{ij}).
\end{equation}
By applying the standard lemma (\ref{1a}) to (\ref{bbb2}) for
$P_{i}$, we get
$$\Phi(TP_{i}+P_{i}T^{*})=\Phi(A_{ij}P_{i}+P_{i}A_{ij}^{*})+\Phi(B_{ij}P_{i}+P_{i}B_{ij}^{*})=\Phi(0)=0,$$
therefore, $\Phi(T_{ii}+T_{ji}+T_{ii}^{*}+T_{ji}^{*})=0$. So
$\Re(T_{ii})=\Re(T_{ji})=0$. Similarly, by using standard lemma
(\ref{2a}) for $P_{i}$, we can obtain $\Im(T_{ii})=\Im(T_{ji})=0$.
Hence $T_{ii}=T_{ji}=0$.\\
Now, we apply standard lemma (\ref{1a}) to (\ref{bbb2}) for $P_{j}$
again, by Claim \ref{1221}, it follows
\begin{eqnarray*}
\Phi(TP_{j}+P_{j}T^{*})&=&\Phi(A_{ij}P_{j}+P_{j}A_{ij}^{*})+\Phi(B_{ij}P_{j}+P_{j}B_{ij}^{*})\\
&=&\Phi(A_{ij}+A_{ij}^{*})+\Phi(B_{ij}+B_{ij}^{*})\\
&=&\Phi(2\Re(A_{ij}))+\Phi(2\Re(B_{ij}^{*}))\\
&=&\Phi(2\Re(A_{ij})+2\Re(B_{ij}^{*}))\\
&=&\Phi(2\Re(A_{ij})+2\Re(B_{ij})).
\end{eqnarray*}
So, we have
$\Phi(T_{ij}+T_{jj}+T_{ij}^{*}+T_{jj}^{*})=\Phi(2\Re(A_{ij}+B_{ij}))$,
or $\Re(T_{ij})=\Re(A_{ij}+B_{ij})$ and $\Re(T_{jj})=0$.\\
Similarly, we can obtain $\Im(T_{ij})=\Im(A_{ij}+B_{ij})$ and
$\Im(T_{jj})=0$. Then $T_{ij}=A_{ij}+B_{ij}$.

\begin{claim}\label{1121}
For every $A_{11}\in\mathcal{A}_{11}$, $C_{21}\in\mathcal{A}_{21}$,
we have
$$\Phi(A_{11}+C_{21})=\Phi(A_{11})+\Phi(C_{21}).$$
\end{claim}

Let $T=T_{11}+T_{12}+T_{21}+T_{22}\in\mathcal{A}$ be such that
\begin{equation}\label{c}
\Phi(T)=\Phi(A_{11})+\Phi(C_{21}).
\end{equation}
By applying the standard lemma (\ref{1a}) to (\ref{c}) for $P_{1}$
and using Claim \ref{1112}, we have
\begin{eqnarray*}
\Phi(TP_{1}+P_{1}T^{*})&=&\Phi(A_{11}P_{1}+P_{1}A_{11}^{*})+\Phi(C_{21}P_{1}+P_{1}C_{21}^{*})\\
&=&\Phi(A_{11}+A_{11}^{*})+\Phi(C_{21}+C_{21}^{*})\\
&=&\Phi(2\Re(A_{11}))+\Phi(2\Re(C_{21}^{*}))\\
&=&\Phi(2\Re(A_{11})+2\Re(C_{21})).
\end{eqnarray*}
Thus,
$\Phi(2\Re(T_{11})+2\Re(T_{21}))=\Phi(2\Re(A_{11})+2\Re(C_{21}))$
which implies $\Re(T_{11})=\Re(A_{11})$ and
$\Re(T_{21})=\Re(C_{21})$. Similarly, we apply the standard lemma
(\ref{2a}) to (\ref{c}) for $P_{1}$, we have
$\Im(T_{11})=\Im(A_{11})$ and $\Im(T_{21})=\Im(C_{21})$. So, $T_{11}=A_{11}$ and $T_{21}=C_{21}$.\\
Now, by applying the standard lemma (\ref{1a}) to (\ref{c}) for
$P_{2}$, it follows
$$\Phi(TP_{2}+P_{2}T^{*})=\Phi(A_{11}P_{2}+P_{2}A_{11}^{*})+\Phi(C_{21}P_{2}+P_{2}C_{21}^{*}).$$
We have $\Phi(2\Re(T_{12})+2\Re(T_{22}))=0$, or
$\Re(T_{12})=\Re(T_{22})=0$. Similarly, we can obtain
$\Im(T_{12})=\Im(T_{22})=0$ by using the standard lemma (\ref{2a})
to (\ref{c}) for $P_{2}$. Then $T_{12}=T_{22}=0$, so we proved
$T=A_{11}+C_{21}$.
\\
\\
Note that $\Phi(B_{12}+D_{22})=\Phi(B_{12})+\Phi(D_{22})$ where
$B_{12}\in\mathcal{A}_{12}$ and $D_{22}\in\mathcal{A}_{22}$ can be
obtained as above.

\begin{claim}\label{11122122}
For every $A_{11}\in\mathcal{A}_{11}$, $B_{12}\in\mathcal{A}_{12}$,
$ {C_{21}}\in\mathcal{A}_{21}$ and $D_{22}\in\mathcal{A}_{22}$ we
have
$$\Phi(A_{11}+B_{12}+C_{21}+D_{22})=\Phi(A_{11})+\Phi(B_{12})+\Phi(C_{21})+\Phi(D_{22}).$$
\end{claim}
Assume $T=T_{11}+T_{12}+T_{21}+T_{22}$ which satisfies in
\begin{equation}\label{d1}
\Phi(T)=\Phi(A_{11})+\Phi(B_{12})+\Phi(C_{21})+\Phi(D_{22}).
\end{equation}
By using the standard lemma ($\ref{1a}$) to ($\ref{d1}$) for $P_{1}$
and Claim \ref{1121}, we obtain
\begin{eqnarray*}
\Phi(TP_{1}+P_{1}T^{*})&=&\Phi(A_{11}P_{1}+P_{1}A_{11}^{*})+\Phi(B_{12}P_{1}+P_{1}B_{12}^{*})\\
&&+\Phi(C_{21}P_{1}+P_{1}C_{21}^{*})+\Phi(D_{22}P_{1}+P_{1}D_{22}^{*})\\
&=&\Phi(2\Re(A_{11})+2\Re(C_{21})).
\end{eqnarray*}
It follows $2\Re(T_{11})+2\Re(T_{21})=2\Re(A_{11})+2\Re(C_{21})$,
hence
$\Re(T_{11})=\Re(A_{11})$ and $\Re(T_{21})=\Re(C_{21})$.\\
Also, one can obtain
$\Im(T_{11})=\Im(A_{11})$ and $\Im(T_{21})=\Im(C_{21})$ by (\ref{2a}).\\
Now, we apply the standard lemma (\ref{1a}) to (\ref{d1}) for
$P_{2}$, it follows
\begin{eqnarray*}
\Phi(TP_{2}+P_{2}T^{*})&=&\Phi(A_{11}P_{2}+P_{2}A_{11}^{*})+\Phi(B_{12}P_{2}+P_{2}B_{12}^{*})\\
&&+\Phi(C_{21}P_{2}+P_{2}C_{21}^{*})+\Phi(D_{22}P_{2}+P_{2}D_{22}^{*})\\
&=&\Phi(2\Re(B_{12})+2\Re(D_{22})).
\end{eqnarray*}
which implies $2\Re(T_{12})+2\Re(T_{22})=2\Re(B_{12})+2\Re(D_{22})$,
therefore $\Re(T_{12})=\Re(B_{12})$ and $\Re(T_{22})=\Re(D_{22})$.
Similarly, we can obtain $\Im(T_{12})=\Im(B_{12})$ and
$\Im(T_{22})=\Im(D_{22})$ by the standard lemma (\ref{2a}). So, we
proved $T=A_{11}+B_{12}+C_{21}+D_{22}$.
\begin{lemma}
Let $\Phi$ satisfy the assumptions of the Main Theorem. Then, for
every $A\in\mathcal{A}$ we have the following
$$\Phi(AI+\lambda IA^{*})=\Phi(A)\Phi(I)+\lambda
\Phi(I)\Phi(A)^{*},$$ where $\lambda\in\{-1,+1\}$.
\end{lemma}
\begin{proof}
By Claim \ref{ijij} and Claim \ref{11122122}, we can obtain
\begin{eqnarray*}
\Phi(AI+\lambda IA^{*})&=&\Phi(A(P_{1}+P_{2})+\lambda
(P_{1}+P_{2})A^{*})\\
&=&\Phi(AP_{1}+AP_{2}+\lambda P_{1}A^{*}+\lambda P_{2}A^{*})\\
&=&\Phi(A_{11}+A_{21}+A_{12}+A_{22}+\lambda A_{11}^{*}+\lambda
A_{21}^{*}+\lambda A_{12}^{*}+\lambda A_{22}^{*})\\
&=&\Phi((A_{11}+\lambda A_{11}^{*})+(A_{12}+\lambda
A_{21}^{*})+(A_{21}+\lambda
A_{12}^{*})+(A_{22}+\lambda A_{22}^{*}))\\
&=&\Phi(A_{11}+\lambda A_{11}^{*})+\Phi(A_{12}+\lambda
A_{21}^{*})+\Phi(A_{21}+\lambda
A_{12}^{*})\\
&&+\Phi(A_{22}+\lambda A_{22}^{*})\\
&=&\Phi(A_{11}+\lambda A_{11}^{*})+\Phi(A_{12})+\Phi(\lambda
A_{21}^{*})+\Phi(A_{21})+\Phi(\lambda
A_{12}^{*})\\
&&+\Phi(A_{22}+\lambda A_{22}^{*})\\
&=&\Phi(A_{11}+\lambda A_{11}^{*}+A_{21}+\lambda A_{21}^{*})+\Phi(A_{22}+\lambda A_{22}^{*}+A_{12}+\lambda A_{12}^{*})\\
&=&\Phi(AP_{1}+\lambda P_{1}A^{*})+\Phi(AP_{2}+\lambda P_{2}A^{*})\\
&=&\Phi(A)\Phi(P_{1})+\lambda
\Phi(P_{1})\Phi(A)^{*}+\Phi(A)\Phi(P_{2})+\lambda\Phi(P_{2})\Phi(A)^{*}\\
&=&
\Phi(A)(\Phi(P_{1})+\Phi(P_{2}))+\lambda(\Phi(P_{1})+\Phi(P_{2}))\Phi(A)^{*}\\
&=&\Phi(A)\Phi(I)+\lambda\Phi(I)\Phi(A)^{*}
\end{eqnarray*}
\end{proof}
\begin{note}\label{note}
By above lemma, we can write the following
\begin{eqnarray*}
\Phi(2\Re(A))&=&\Phi(AI+IA^{*})\\
&=&\Phi(A)\Phi(I)+\Phi(I)\Phi(A)^{*}\\
&=&\Phi(A)+\Phi(A)^{*}\\
&=&2\Re(\Phi(A)).
\end{eqnarray*}
And similarly, we have
\begin{eqnarray*}
\Phi(2i\Im(A))&=&\Phi(AI-IA^{*})\\
&=&\Phi(A)\Phi(I)-\Phi(I)\Phi(A)^{*}\\
&=&\Phi(A)-\Phi(A)^{*}\\
&=&2i\Im(\Phi(A)).
\end{eqnarray*}
\end{note}

\begin{lemma}\label{p}
Let $\Phi$ satisfy the assumptions of the Main Theorem, we have the
following
\begin{equation}\label{eq1}
\Phi(AP_{i})=\Phi(A)\Phi(P_{i}),
\end{equation}
and
\begin{equation}\label{eq2}
\Phi(P_{i}A)=\Phi(P_{i})\Phi(A),
\end{equation}
for $1\leq i\leq 2$.
\end{lemma}
\begin{proof}
We only prove the equation (\ref{eq1}).\\
Let $A=\sum_{i,j=1}^{2}A_{ij}$, by assumption of the Main Theorem
together, we have
$$\Phi(AP_{i}+P_{i}A^{*})=\Phi(A)\Phi(P_{i})+\Phi(P_{i})\Phi(A)^{*},$$
$$\Phi(AP_{i}-P_{i}A^{*})=\Phi(A)\Phi(P_{i})-\Phi(P_{i})\Phi(A)^{*}.$$
Add these two equations together, by Claim \ref{1221}, Claim
\ref{1121} and Note \ref{note}, we have
\begin{eqnarray*}
2\Phi(A)\Phi(P_{i})&=&\Phi(AP_{i}+P_{i}A^{*})+\Phi(AP_{i}-P_{i}A^{*})\\
&=&\Phi(A_{ii}+A_{ji}+A_{ii}^{*}+A_{ji}^{*})+\Phi(A_{ii}+A_{ji}-A_{ii}^{*}-A_{ji}^{*})\\
&=&\Phi(A_{ii}+A_{ii}^{*})+\Phi(A_{ji}+A_{ji}^{*})+\Phi(A_{ii}-A_{ii}^{*})\\
&&+\Phi(A_{ji}-A_{ji}^{*})\\
&=&\Phi(2\Re(A_{ii}))+\Phi(2i\Im(A_{ii}))+\Phi(A_{ji})+\Phi(A_{ji}^{*})\\
&&+\Phi(A_{ji})-\Phi(A_{ji}^{*})\\
&=&2\Re(\Phi(A_{ii}))+2i\Im(\Phi(A_{ii}))+2\Phi(A_{ji})\\
&=&2\Phi(A_{ii}+A_{ji})\\
&=&2\Phi(AP_{i}).
\end{eqnarray*}

\end{proof}
\begin{claim}\label{projection}
$\Phi$ preserves projections $P_{i}$ $(i=1,2)$ in both directions.
\end{claim}
Let $P_{i}$  be  projections, we have
$$2\Phi(P_{i})=\Phi(I)\Phi(P_{i})+\Phi(P_{i})\Phi(I)^{*}=\Phi(IP_{i}+P_{i}I^{*})=\Phi(2P_{i}).$$
Also,
$$\Phi(2P_{i})=\Phi(P_{i}+P_{i}^{*})=\Phi(P_{i})\Phi(I)+\Phi(I)\Phi(P_{i})^{*},$$
we can write $2\Phi(P_{i})=\Phi(P_{i})+\Phi(P_{i})^{*}$. So,
$\Phi(P_{i})=\frac{\Phi(P_{i})+\Phi(P_{i})^{*}}{2}$. Let
$Q_{i}=\frac{\Phi(P_{i})+\Phi(P_{i})^{*}}{2}$, therefore, $\Phi(P_{i})=Q_{i}$.\\
On the other hand, we have the following
$$\Phi(P_{i}P_{i}+P_{i}P_{i}^{*})=\Phi(P_{i})\Phi(P_{i})+\Phi(P_{i})\Phi(P_{i})^{*}.$$
From above equation, we have $\Phi(2P_{i})=2Q_{i}^{2}.$ So,
$2\Phi(P_{i})=\Phi(2P_{i})=2Q_{i}^{2}$, hence we can say $Q_{i}=Q_{i}^{2}$.\\
Conversely, assume that $\Phi(P_{i})$ is a projection. Since
$\Phi^{-1}$ has the same property as $\Phi$ has, a similar
discussion implies that $P_{i}$ is a projection.
\\

Now, Claim \ref{projection} ensures that there exist nontrivial
projections $Q_{i}$ $(i=1,2)$ such that $\Phi(P_{i})=Q_{i}$. By
Claim \ref{11122122}, $Q_{1}+Q_{2}=I$. We can write
$\mathcal{B}=\sum_{i,j=1}^{2}\mathcal{B}_{ij}$ where
$\mathcal{B}_{ij}=Q_{i}\mathcal{B}Q_{j},\ i,j=1,2$.

\begin{claim}\label{cl6}
$\Phi(\mathcal{A}_{ij})=\mathcal{B}_{ij}$
\end{claim}
As we assumed, let $P_{i}\in\mathcal{A}_{ii}$ and
$\Phi(P_{i})=Q_{i}\in\mathcal{B}_{ii}$. Since $\Phi$ is surjective
we have $\Phi(\mathcal{A})=\mathcal{B}$. By multiplying the left and
the right side of latter equation by $\Phi(P_{i})$ and $\Phi(P_{j})$
respectively, we have
$$\Phi(P_{i})\Phi(\mathcal{A})\Phi(P_{j})=\Phi(P_{i})\mathcal{B}\Phi(P_{j})\subseteq
\mathcal{B}_{ij}.$$ So, $\Phi(\mathcal{A}_{ij})\subseteq
\mathcal{B}_{ij}$. Similarly, we can prove
$\Phi(\mathcal{A}_{ii})\subseteq \mathcal{B}_{ii},$
$\Phi(\mathcal{A}_{ji})\subseteq \mathcal{B}_{ji},$
$\Phi(\mathcal{A}_{jj})\subseteq \mathcal{B}_{jj}$.
 For the
converse, because of surjectivity there are some members in
$\mathcal{B}_{ij}$ which should be covered by members of
$\mathcal{A}$. On the other hand, we know that each partition can
cover its corresponding partition. So,
$\Phi(\mathcal{A}_{ij})=\mathcal{B}_{ij}$.

We should mention here that we imply the primeness property just in
this claim.
\begin{claim}\label{iijj}
For every $A_{ii}, B_{ii}\in \mathcal{A}_{ii}$, $1\leq i\leq 2$ we
have
$$\Phi(A_{ii}+B_{ii})=\Phi(A_{ii})+\Phi(B_{ii}).$$
\end{claim}
First, we will prove that
$\Phi(AP_{i}+BP_{i})=\Phi(AP_{i})+\Phi(BP_{i})$ for every
$A,B\in \mathcal{A}$.\\
By Lemma \ref{p}, Claim \ref{ijij} and for every
$T_{ij}\in\mathcal{A}$ such that $i\neq j$ we obtain
\begin{eqnarray*}
\Phi(T_{ij})\Phi(AP_{i}+BP_{i})&=&\Phi(P_{i}T)\Phi(P_{j})\Phi(AP_{i}+BP_{i})\\
&=&\Phi(P_{i}T)\Phi(P_{j}AP_{i}+P_{j}BP_{i})\\
&=&\Phi(P_{i}T)\Phi(A_{ji}+B_{ji})\\
&=&\Phi(P_{i}T)\Phi(A_{ji})+\Phi(P_{i}T)\Phi(B_{ji})\\
&=&\Phi(P_{i}TP_{j})\Phi(AP_{i})+\Phi(P_{i}TP_{j})\Phi(BP_{i})\\
&=&\Phi(T_{ij})(\Phi(AP_{i})+\Phi(BP_{i})).
\end{eqnarray*}
By the primeness of $\mathcal{B}$ and Claim \ref{cl6}, we have
\begin{equation}\label{eq5}
\Phi(AP_{i}+BP_{i})=\Phi(AP_{i})+\Phi(BP_{i}).
\end{equation}
 Now, multiply the
left side of equation (\ref{eq5}) by $\Phi(P_{i})$ and use Lemma
\ref{p}, we obtain
$$\Phi(P_{i}AP_{i}+P_{i}BP_{i})=\Phi(P_{i}AP_{i})+\Phi(P_{i}BP_{i}).$$
\\
So, additivity of $\Phi$ comes from Claim \ref{ijij},
\ref{11122122}, \ref{iijj}.
\\
\\
Since $\Phi$ is additive, we have
$$\Phi(A+A^{*})=\Phi(A)+\Phi(A^{*}),$$
also, by (\ref{1a}), it follows
$$\Phi(A+A^{*})=\Phi(A)+\Phi(A)^{*}.$$
So, by two equations we can write $\Phi(A^{*})=\Phi(A)^{*}$.

\end{document}